\documentclass[12pt]{amsart}
\usepackage{amscd,amssymb,amsmath,multicol,scalefnt,etex,tikz}
\newtheorem{thm}[equation]{Theorem}
\numberwithin{equation}{section}
\newtheorem{cor}[equation]{Corollary}

\newtheorem{lem}[equation]{Lemma}

\newtheorem{prop}[equation]{Proposition}

\begin{document}
\raggedbottom \voffset=-.7truein \hoffset=0truein \vsize=8truein
\hsize=6truein \textheight=8truein \textwidth=6truein
\baselineskip=18truept

\def\mapright#1{\ \smash{\mathop{\longrightarrow}\limits^{#1}}\ }
\def\mapleft#1{\smash{\mathop{\longleftarrow}\limits^{#1}}}
\def\mapup#1{\Big\uparrow\rlap{$\vcenter {\hbox {$#1$}}$}}
\def\mapdown#1{\Big\downarrow\rlap{$\vcenter {\hbox {$\ssize{#1}$}}$}}
\def\mapne#1{\nearrow\rlap{$\vcenter {\hbox {$#1$}}$}}
\def\mapse#1{\searrow\rlap{$\vcenter {\hbox {$\ssize{#1}$}}$}}
\def\mapr#1{\smash{\mathop{\rightarrow}\limits^{#1}}}
\def\ss{\smallskip}
\def\ar{\arrow}
\def\vp{v_1^{-1}\pi}
\def\at{{\widetilde\alpha}}
\def\sm{\wedge}
\def\la{\langle}
\def\ra{\rangle}
\def\on{\operatorname}
\def\ol#1{\overline{#1}{}}
\def\spin{\on{Spin}}
\def\cat{\on{cat}}
\def\lbar{\ell}
\def\qed{\quad\rule{8pt}{8pt}\bigskip}
\def\ssize{\scriptstyle}
\def\a{\alpha}
\def\bz{{\Bbb Z}}
\def\Rhat{\hat{R}}
\def\im{\on{im}}
\def\ct{\widetilde{C}}
\def\ext{\on{Ext}}
\def\sq{\on{Sq}}
\def\eps{\epsilon}
\def\ar#1{\stackrel {#1}{\rightarrow}}
\def\br{{\bold R}}
\def\bC{{\bold C}}
\def\bA{{\bold A}}
\def\bB{{\bold B}}
\def\bD{{\bold D}}
\def\bh{{\bold H}}
\def\bQ{{\bold Q}}
\def\bP{{\bold P}}
\def\bx{{\bold x}}
\def\bo{{\bold{bo}}}
\def\si{\sigma}
\def\Vbar{{\overline V}}
\def\dbar{{\overline d}}
\def\wbar{{\overline w}}
\def\Sum{\sum}
\def\tfrac{\textstyle\frac}
\def\tb{\textstyle\binom}
\def\Si{\Sigma}
\def\w{\wedge}
\def\equ{\begin{equation}}
\def\AF{\operatorname{AF}}
\def\b{\beta}
\def\G{\Gamma}
\def\D{\Delta}
\def\L{\Lambda}
\def\g{\gamma}
\def\k{\kappa}
\def\psit{\widetilde{\Psi}}
\def\tht{\widetilde{\Theta}}
\def\psiu{{\underline{\Psi}}}
\def\thu{{\underline{\Theta}}}
\def\aee{A_{\text{ee}}}
\def\aeo{A_{\text{eo}}}
\def\aoo{A_{\text{oo}}}
\def\aoe{A_{\text{oe}}}
\def\vbar{{\overline v}}
\def\endeq{\end{equation}}
\def\sn{S^{2n+1}}
\def\zp{\bold Z_p}
\def\cR{{\mathcal R}}
\def\P{{\mathcal P}}
\def\cF{{\mathcal F}}
\def\cQ{{\mathcal Q}}
\def\cj{{\cal J}}
\def\zt{{\bold Z}_2}
\def\bs{{\bold s}}
\def\bof{{\bold f}}
\def\bq{{\bold Q}}
\def\be{{\bold e}}
\def\Hom{\on{Hom}}
\def\ker{\on{ker}}
\def\kot{\widetilde{KO}}
\def\coker{\on{coker}}
\def\da{\downarrow}
\def\colim{\operatornamewithlimits{colim}}
\def\zphat{\bz_2^\wedge}
\def\io{\iota}
\def\Om{\Omega}
\def\Prod{\prod}
\def\e{{\cal E}}
\def\zlt{\Z_{(2)}}
\def\exp{\on{exp}}
\def\abar{{\overline a}}
\def\xbar{{\overline x}}
\def\ybar{{\overline y}}
\def\zbar{{\overline z}}
\def\Rbar{{\overline R}}
\def\nbar{{\overline n}}
\def\cbar{{\overline c}}
\def\qbar{{\overline q}}
\def\bbar{{\overline b}}
\def\et{{\widetilde E}}
\def\ni{\noindent}
\def\coef{\on{coef}}
\def\den{\on{den}}
\def\lcm{\on{l.c.m.}}
\def\vi{v_1^{-1}}
\def\ot{\otimes}
\def\psibar{{\overline\psi}}
\def\thbar{{\overline\theta}}
\def\mhat{{\hat m}}
\def\exc{\on{exc}}
\def\ms{\medskip}
\def\ehat{{\hat e}}
\def\etao{{\eta_{\text{od}}}}
\def\etae{{\eta_{\text{ev}}}}
\def\dirlim{\operatornamewithlimits{dirlim}}
\def\gt{\widetilde{L}}
\def\lt{\widetilde{\lambda}}
\def\st{\widetilde{s}}
\def\ft{\widetilde{f}}
\def\sgd{\on{sgd}}
\def\lfl{\lfloor}
\def\rfl{\rfloor}
\def\ord{\on{ord}}
\def\gd{{\on{gd}}}
\def\rk{{{\on{rk}}_2}}
\def\nbar{{\overline{n}}}
\def\MC{\on{MC}}
\def\lg{{\on{lg}}}
\def\cB{\mathcal{B}}
\def\cS{\mathcal{S}}
\def\cP{\mathcal{P}}
\def\N{{\Bbb N}}
\def\Z{{\Bbb Z}}
\def\Q{{\Bbb Q}}
\def\R{{\Bbb R}}
\def\C{{\Bbb C}}
\def\F{{\Bbb F}}
\def\l{\left}
\def\r{\right}
\def\mo{\on{mod}}
\def\xt{\times}
\def\notimm{\not\subseteq}
\def\Remark{\noindent{\it  Remark}}
\def\kut{\widetilde{KU}}
\def\tz{tikzpicture}
\def\*#1{\mathbf{#1}}
\def\0{$\*0$}
\def\1{$\*1$}
\def\22{$(\*2,\*2)$}
\def\33{$(\*3,\*3)$}
\def\ss{\smallskip}
\def\ssum{\sum\limits}
\def\dsum{\displaystyle\sum}
\def\la{\langle}
\def\ra{\rangle}
\def\on{\operatorname}
\def\od{\text{od}}
\def\ev{\text{ev}}
\def\o{\on{o}}
\def\U{\on{U}}
\def\lg{\on{lg}}
\def\a{\alpha}
\def\bz{{\Bbb Z}}
\def\eps{\varepsilon}
\def\bc{{\bold C}}
\def\bN{{\bold N}}
\def\nut{\widetilde{\nu}}
\def\tfrac{\textstyle\frac}
\def\b{\beta}
\def\G{\Gamma}
\def\g{\gamma}
\def\zt{{\Bbb Z}_2}
\def\pt{\widetilde{p}}
\def\zth{{\bold Z}_2^\wedge}
\def\bs{{\bold s}}
\def\bx{{\bold x}}
\def\bof{{\bold f}}
\def\bq{{\bold Q}}
\def\be{{\bold e}}
\def\lline{\rule{.6in}{.6pt}}
\def\xb{{\overline x}}
\def\xbar{{\overline x}}
\def\ybar{{\overline y}}
\def\zbar{{\overline z}}
\def\ebar{{\overline \be}}
\def\nbar{{\overline n}}
\def\rbar{{\overline r}}
\def\Mbar{{\overline M}}
\def\et{{\widetilde e}}
\def\ni{\noindent}
\def\ms{\medskip}
\def\ehat{{\hat e}}
\def\what{{\widehat w}}
\def\Yhat{{\widehat Y}}
\def\nbar{{\overline{n}}}
\def\minp{\min\nolimits'}
\def\mul{\on{mul}}
\def\N{{\Bbb N}}
\def\Z{{\Bbb Z}}
\def\Q{{\Bbb Q}}
\def\R{{\Bbb R}}
\def\C{{\Bbb C}}
\def\notint{\cancel\cap}
\def\se{\operatorname{secat}}
\def\cS{\mathcal S}
\def\cR{\mathcal R}
\def\el{\ell}
\def\TC{\on{TC}}
\def\dstyle{\displaystyle}
\def\ds{\dstyle}
\def\mt{\widetilde{\mu}}
\def\zcl{\on{zcl}}
\def\Vb#1{{\overline{V_{#1}}}}
\def\tz{tikzpicture}

\def\Remark{\noindent{\it  Remark}}
\title[$BP$-homology of elementary 2-groups]
{$BP$-homology of elementary abelian 2-groups: $BP$-module structure}
\author{Donald M. Davis}
\address{Department of Mathematics, Lehigh University\\Bethlehem, PA 18015, USA}
\email{dmd1@lehigh.edu}
\date{June 25, 2018}


\maketitle
\begin{abstract} We determine the $BP_*$-module structure, mod higher filtration, of the main part of the $BP$-homology of elementary abelian 2-groups. The action is related to symmetric polynomials and to Dickson invariants.
 \end{abstract}
\section{Introduction and results}\label{intro} Let $BP_*(-)$ denote  Brown-Peterson homology localized at 2. Its coefficient groups $BP_*$ are a polynomial algebra over $\Z_{(2)}$ on classes $v_j$, $j\ge1$, of grading $2(2^j-1)$. Let $v_0=2$.  As was done in \cite{JW} and \cite{JWY}, we consider
$\bigotimes_{BP_*}^k BP_*(B\Z/2)$, which is a $BP_*$-direct summand of $BP_*(B(\Z/2)^k)$. We determine the $BP_*$-module structure of $\bigotimes_{BP_*}^k BP_*(B\Z/2)$ modulo terms which are more highly divisible by $v_j$'s. Information about the action of $v_0$ was applied to problems in topology in \cite{D2} and \cite{SW}. In the forthcoming paper \cite{DTC}, we apply it to another problem, higher topological complexity of real projective spaces. In Theorem \ref{thm4} of the current paper, we obtain complete explicit information about the $v_0$-action (mod higher filtration). In Theorem \ref{thm1}, we determine the action of all $v_j$'s as quotients of symmetric polynomials, and in Theorem \ref{thm2} and Corollary \ref{BPkcor} we give explicit formulas as symmetric polynomials in certain families of cases. In Section \ref{Dicksec}, we discuss relationships of our symmetric polynomials with the Dickson invariants.

Now we explain this more explicitly.
There are $BP_*$-generators $z_i\in BP_{2i-1}(B\Z/2)$ for $i\ge1$, and $\bigotimes_{BP_*}^k BP_*(B\Z/2)$ is spanned as a $BP_*$-module by classes $z_I=z_{i_1}\ot\cdots\ot z_{i_k}$ for $I=(i_1,\ldots,i_k)$ with $i_j\ge1$. Let $Z_k$ denote the graded set consisting of all such classes $z_I$. It was proved in \cite[Thm 3.2]{JW} that $\bigotimes_{BP_*}^k BP_*(B\Z/2)$ admits a decreasing filtration by
$BP_*$-submodules $F_s$ such that, for $s\ge0$, the quotient $F_s/F_{s+1}$ is a vector space over the prime field $\F_2$ with basis all classes $(v_k^{t_k}v_{k+1}^{t_{k+1}}\cdots)z_I$ with $z_I\in Z_k$, $t_i\ge0$, and $\sum t_i=s$.

Define an action of $\F_2[x_1,\ldots,x_k]$ on the $\F_2$-vector space with basis $Z_k$ by $$x_1^{e_1}\cdots x_k^{e_k}\cdot z_I=z_{I-E},$$ where $I-E=(i_1-e_1,\ldots,i_k-e_k)$; here, by convention, $z_J=0$ if any entry of $J$ is $\le0$. For positive integers $t_1,\ldots,t_r$, let $m_{t_1,\ldots,t_r}$ denote the monomial symmetric polynomial in $x_1,\ldots,x_k$, the smallest symmetric polynomial containing the monomial $x_1^{t_1}\cdots x_r^{t_r}$. Over $\F_2$, if $r=k$ and the $t_i$ are distinct, it equals the Vandermonde determinant
$$\begin{vmatrix}x_1^{t_1}&\cdots&x_1^{t_k}\\ &\vdots&\\ x_k^{t_1}&\cdots&x_k^{t_k}\end{vmatrix}.$$

Our first theorem determines the action of $v_j$, $0\le j\le k-1$, from $F_s/F_{s+1}$ to $F_{s+1}/F_{s+2}$, as a ratio of monomial symmetric polynomials in $x_1,\ldots,x_k$. Note that $k$ is fixed throughout, and we are always dealing with polynomials over $\F_2$.
This theorem will be proved in Section \ref{sec2}.
\begin{thm}\label{thm1} If $F_s$ is as above, and $0\le j\le k-1$, the action of $v_j$ from $F_s/F_{s+1}$ to $F_{s+1}/F_{s+2}$ is multiplication by $\ds\sum_{\ell\ge k}v_\ell p_{\ell,j}$, where \begin{equation}\label{pdef}p_{\ell,j}=\frac{m_{2^0,\ldots,\widehat{2^j},\ldots,2^{k-1},2^\ell}}{m_{2^0,\ldots,2^{k-1}}}.\end{equation}
(The $\widehat {2^j}$ notation denotes omission.) Moreover, $p_{\ell,j}$ is a symmetric polynomial, mod $2$.
\end{thm}

It is not {\it a priori} clear that the quotient on the right hand side of (\ref{pdef}) should be a polynomial mod 2. In fact, if the $2^\ell$ there is a replaced by a non-2-power and $k\ge3$, then the ratio is not a polynomial mod 2.

We have obtained explicit polynomial formulas for $p_{\ell,j}$ in several cases. These will be proved in Section \ref{compsec}. The first is the complete solution when $k=3$.
\begin{thm}\label{thm2} If $k=3$ and $\ell\ge3$, then
\begin{eqnarray*}p_{\ell,0}&=&\sum_{\substack{i\ge j\ge k>0\\ i+j+k=2^\ell-1}}\tbinom{j+k}k m_{i,j,k}\\
p_{\ell,1}&=&\sum_{\substack{i\ge j>0\\ i+j=2^\ell-2}}(1+j)m_{i,j,0}+\sum_{\substack{i\ge j\ge k>0\\ i+j+k=2^\ell-2}}(1+\tbinom{j+k}{k-1}+\tbinom{j+k+1}{k+1})m_{i,j,k}\\
p_{\ell,2}&=&\sum_{\substack{i\ge j\ge k\ge0\\ i+j+k=2^{\ell}-4}}(1+\tbinom{j+k+2}{k+1})m_{i,j,k}.\end{eqnarray*}
\end{thm}

Incorporating Theorem \ref{thm2} into Theorem \ref{thm1} gives the $v_0$-, $v_1$-, and $v_2$-action, mod higher filtration, in $BP_*(B\Z/2)\ot_{BP_*} BP_*(B\Z/2)\ot_{BP_*}BP_*(B\Z/2)$. For example, $v_0$ acts as
\begin{equation}\label{v0}v_3m_{4,2,1}+v_4(m_{12,2,1}+m_{10,4,1}+m_{8,6,1}+m_{9,4,2}+m_{8,5,2}+m_{8,4,3})+\cdots,\end{equation}
where the omitted terms involve $v_\ell$ for $\ell\ge5$.

We have also obtained the explicit polynomial formula for (\ref{pdef}) for any $k$ if  $\ell=k$.
\begin{thm}\label{thm3} If $\ell=k$, then $p_{\ell,j}=p_{k,j}$ equals the sum of all monomials of degree $2^k-2^j$ in $x_1,\ldots,x_k$ in which all nonzero exponents are $2$-powers. Here $0\le j\le k-1$.\end{thm}

Theorem \ref{thm3} gives the formula for the $v_k$-component of the $BP_*$-module structure, modulo higher filtration, of $\bigotimes_{BP_*}^k BP_*(B\Z/2)$. It is complete information, mod higher filtration, for $BP\la k\ra$ homology.
 Johnson-Wilson homology $BP\la k\ra$, introduced in \cite{JW1}, has coefficients $\Z_{(2)}[v_1,\ldots,v_k]$. By \cite[Thm 3.2]{JW} and \cite[Thm 1.1]{JWY}, as an abelian group $\bigotimes^k_{BP\la k\ra_*}BP\la k\ra_*(B\Z/2)$ has basis $\{v_k^jz_I:\ j\ge0,\ z_I\in Z_k\}$.
\begin{cor} In $\bigotimes^k_{BP\la k\ra_*}BP\la k\ra_*(B\Z/2)$, for $0\le j\le k-1$,
$$v_j\cdot z_I\equiv v_k\sum_Ez_{I-E}$$
mod higher filtration, where $E=(e_1,\ldots,e_k)$ ranges over all $k$-tuples such that all nonzero $e_j$ are $2$-powers, and $\sum e_j=2^k-2^j$.\label{BPkcor}\end{cor}
\noindent This generalizes \cite[Cor 2.7]{JWY}, which says roughly that  $v_0$ acts as $v_km_{2^{k-1},2^{k-2},\ldots,1}$.

Finally, our most elaborate, and probably most useful, explicit calculation is given in the following result, which gives the complete formula for the $v_0$-action, mod higher filtration. This is useful since $v_0$ corresponds to multiplication by 2.
\begin{thm}\label{thm4} In $\bigotimes_{BP_*}^kBP_*(B\Z/2)$, $v_0$ acts as $\ds\sum_{\ell\ge k}v_\ell\cdot p_{\ell,0}$ mod higher filtration, where
$$p_{\ell,0}=\sum_f\prod_{i=0}^{\ell-1}x_{f(i)}^{2^i},$$
where $f$ ranges over all surjective functions $\{0,\ldots,\ell-1\}\to\{1,\ldots,k\}$. Equivalently, $p_{\ell,0}=\sum m_{\|S_1\|,\ldots,\|S_k\|}$, where the sum ranges over all $\|S_1\|>\cdots>\|S_k\|$ with $S_1,\ldots,S_k$ a partition of $\{1,2,4,\ldots,2^{\ell-1}\}$ into $k$ nonempty subsets. Here $\|S\|$ is the sum of the elements of $S$.\end{thm}
\noindent See (\ref{v0}) for an explicit example of $p_{3,0}$ and $p_{4,0}$ when $k=3$. For example, the term $m_{10,4,1}$ in $p_{4,0}$ corresponds to $S_1=\{8,2\}$, $S_2=\{4\}$, and $S_3=\{1\}$, and this corresponds to the sum of all surjective functions $f:\{0,1,2,3\}\to\{1,2,3\}$ for which  $f(3)=f(1)$.

 We thank a referee for  many  useful suggestions. See especially Section \ref{Dicksec}.

\section{Proof of Theorem \ref{thm1}} \label{sec2}
In this section, we prove Theorem \ref{thm1}.
\begin{proof}[Proof of Theorem \ref{thm1}]
Let $Q=\bigotimes^k_{BP_*}BP_*(B\Z/2)$. Let $z_i$ and $z_I$ be as in the second paragraph of the paper. By \cite{JW}, $Q$ is spanned by classes $(v_0^{t_0}v_1^{t_1}\cdots)z_I$ with only relations $\sum_{j\ge0}a_jz_{i-j}$ in any factor, where $a_j\in BP_{2j}$ are coefficients in the [2]-series. By \cite[3.17]{W}, these satisfy, mod $(v_0,v_1,\ldots)^2$,
$$a_j\equiv\begin{cases}v_i&j=2^i-1,\  i\ge0\\ 0&j+1\text{ not a 2-power.}\end{cases}$$
Let $F_s$ denote the ideal $(v_0,v_1,\ldots)^sQ$. Then $F_s/F_{s+1}$ is spanned by all
$(v_0^{t_0}v_1^{t_1}\cdots)z_I$ with $\sum t_j=s$, with relations \begin{equation}\label{vshort}\sum_{j\ge0}v_jz_{i-(2^j-1)}=0\end{equation}
in each factor. As proved in \cite[Thm 3.2]{JW} (see also \cite[2.3]{JWY}), this leads to an $\F_2$-basis for $F_s/F_{s+1}$ consisting of all $(v_k^{t_k}v_{k+1}^{t_{k+1}}\cdots)z_I$ with $\sum t_j=s$.

We claim that if $z_I\in F_0$ and $0\le j\le k-1$, then we must have
\begin{equation}\label{veq} v_jz_I=\sum_{\ell\ge k}v_\ell p_{\ell,j}z_I,\end{equation}
where $p_{\ell,j}$ is a symmetric polynomial in variables $x_1,\ldots,x_k$ of degree $2^\ell-2^j$, acting on $z_I$ by decreasing subscripts as described in the third paragraph of the paper. That the action is symmetric and uniform is due to the uniform nature of the relations (\ref{vshort}). That it never increases subscripts of $z_i$ is a consequence of naturality: there are inclusions $\bigotimes_{BP_*} BP_*(RP^{2n_i})\to\bigotimes^k_{BP_*} BP_*(B\Z/2)$ in which the only $z_I$ in the image are those with $i_t\le n_t$ for all $t$, and the $v_j$-actions are compatible.

Note that (\ref{vshort}) can be interpreted as saying that, for any $i\in\{1,\ldots,k\}$,
\begin{equation}\label{vr}\sum_{j\ge0} v_jx_i^{2^j-1}=0.\end{equation}
Since the $v_\ell$-components are independent if $\ell\ge k$, and (\ref{veq}) says that for $j<k\le \ell$ the $v_\ell$-component of the $v_j$-action is given by the (unknown) polynomial $p_{\ell,j}$, we obtain the equation
$$\sum_{j=0}^{k-1}p_{\ell,j}x_i^{2^j-1}=x_i^{2^\ell-1}$$
for any $i\in\{1,\ldots,k\}$ and $\ell\ge k$. After multiplying the $i$th equation by $x_i$, we obtain the system
\begin{equation}\label{eq}\begin{bmatrix}x_1&x_1^2&x_1^4&\cdots&x_1^{2^{k-1}}\\
&&\vdots&&\\
x_k&x_k^2&x_k^4&\cdots&x_k^{2^{k-1}}\end{bmatrix}
\begin{bmatrix}p_{\ell,0}\\ \vdots\\ p_{\ell,k-1}\end{bmatrix}=\begin{bmatrix}x_1^{2^\ell}\\ \vdots\\ x_k^{2^\ell}\end{bmatrix},\end{equation}
 whose solution as (\ref{pdef}) is given by Cramer's Rule. Our argument shows  that the components $p_{\ell,j}$ of the solution are polynomials, mod 2.

 The ratios on the RHS of (\ref{pdef}) can also be shown to be polynomials by the following algebraic argument, provided by the referee. Let $V$ denote the $\F_2$-vector space with basis $x_1,\ldots,x_k$. The denominator $m_{1,2,\ldots,2^{k-1}}$ in (\ref{pdef}) equals the product of the nonzero elements $v$ of $V$. We show that a Vandermonde determinant $D$ in $x_1,\ldots,x_k$ with distinct 2-power exponents $2^{t_j}$ is divisible by each $v$ in the unique factorization domain $\F_2[x_1,\ldots,x_k]$, and hence is divisible by their product.

 By induction on $k$ and expansion along rows, the determinant is divisible by all elements except perhaps $\ds\sum_{i=1}^k x_i$. Let $M_{k,j}$ denote the minor associated with $x_k^{2^{t_j}}$. Replacing the last row by the sum of the others shows that
 $$\sum_jM_{k,j}\sum_{i=1}^{k-1}x_i^{2^{t_j}}=0$$ since it is the determinant of a matrix with dependent rows. Thus
 $$D=\sum_j M_{k,j}x_k^{2^{t_j}}=\sum_jM_{k,j}\sum_{i=1}^kx_i^{2^{t_j}}=\sum_jM_{k,j}\biggl(\sum_{i=1}^kx_i\biggr)^{2^{t_j}}$$
 is divisible by $\ds\sum_{i=1}^kx_i$.

The $v_j$-action formula on $F_0$ applies also on $F_s$ by the nature of the module.
\end{proof}

\section{Proofs of explicit formulas for certain $p_{\ell,j}$}\label{compsec}
In this section, we prove Theorems \ref{thm2}, \ref{thm3}, and \ref{thm4}.
\begin{proof}[Proof of Theorem \ref{thm2}] Let $h_d(x_1,\ldots,x_r)$ denote the complete homogeneous polynomial of degree $d$.
With $k=3$, after  a few row operations, (\ref{eq}) reduces to
$$\begin{bmatrix}1&x_1&x_1^3\\ 0&1&h_2(x_1,x_2)\\ 0&0&x_1+x_2+x_3\end{bmatrix}\begin{bmatrix}p_{\ell,0}\\
p_{\ell,1}\\ p_{\ell,2}\end{bmatrix}=\begin{bmatrix}x_1^{2^\ell-1}\\ h_{2^\ell-2}(x_1,x_2)\\ h_{2^\ell-3}(x_1,x_2,x_3)\end{bmatrix}.$$

Using Pascal's formula, one easily verifies, mod 2,
$$h_{n+1}(x_1,x_2,x_3)\equiv(x_1+x_2+x_3)\sum_{k,j}\bigl(\tbinom{n+2-k}{j+1}-1\bigr)x_1^{n-j-k}x_2^jx_3^k.$$
Since $\binom{2^\ell-2-k}{j+1}\equiv\binom{j+k+2}{j+1}$, the result for $p_{\ell,2}$ follows.

Now we have
\begin{eqnarray*}p_{\ell,1}&=&h_{2^\ell-2}(x_1,x_2)-h_2(x_1,x_2)p_{\ell,2}\\
&=&\sum x_1^ix_2^{2^\ell-2-i}+(x_1^2+x_1x_2+x_2^2)\sum_{\substack{i\ge j\ge k\ge0\\ i+j+k=2^\ell-4}}(1+\tbinom{j+k+2}{k+1})m_{i,j,k}.\end{eqnarray*}
If $k>0$, the coefficient of $m_{i,j,k}$ in this is
$$(1+\tbinom{j+k+2}{k+1})+(1+\tbinom{j+k+1}{k+1})+(1+\tbinom{j+k}{k+1}),$$ which equals the claimed value. If $k=0$ and $j>0$, there is an extra 1 from the $\sum x_1^ix_2^{2^\ell-2-i}$, and we obtain $\tbinom{j+2}1+\tbinom{j+1}1+\tbinom j1\equiv 1+j$, as desired. The coefficient of $m_{2^\ell-4,0,0}$ is easily seen to be 0.

Finally, we obtain $p_{\ell,0}$ from $x_1^{2^\ell-1}+x_1p_{\ell,1}+x_1^3p_{\ell,2}$. The coefficient of $m_{i,j,0}$ in this is $(1+j)+(1+\tbinom{j+2}1)=0$, as desired. If $k>0$, the coefficient of $m_{i,j,k}$ is $(1+\tbinom{j+k}{k-1}+\tbinom{j+k+1}{k+1})+(1+\tbinom{j+k+2}{k+1})\equiv\tbinom{j+k}k$, as desired.
\end{proof}
\begin{proof}[Proof of Theorem \ref{thm3}] It suffices to show that
\begin{equation}\label{eq2}\sum_{i=1}^k x_1^{2^{i-1}}g_{2^k-2^{i-1}}=x_1^{2^k},\end{equation}
where $g_m$ is the sum of all monomials in $x_1,\ldots,x_k$ of degree $m$ with all nonzero exponents $2$-powers. (Other rows are handled equivalently.)

The term $x_1^{2^k}$ is obtained once, when $i=k$. The only monomials obtained in the LHS of (\ref{eq2})
have their $x_i$-exponent a $2$-power for $i>1$, while their $x_1$-exponent may be a $2$-power or the sum of two distinct 2-powers. A term of the first type, $x_1^{2^i}x_2^{2^{t_2}}\cdots x_k^{2^{t_k}}$ with $\sum 2^{t_i}>0$, can be obtained from either the $i$th term in (\ref{eq2}) or the $(i+1)$st. So its coefficient is 0 mod 2. A term of the second type, $x_1^{2^a+2^b}x_2^{t_2}\cdots x_k^{t_k}$, can also be obtained in two ways, either from $i=a+1$ or $i=b+1$.\end{proof}

Theorem \ref{thm4} is an immediate consequence of the following proposition, which shows that, in $\F_2[x_1,\ldots,x_k]$,
$$m_{2^1,\ldots,2^{k-1},2^\ell}=m_{2^0,\ldots,2^{k-1}}\cdot\sum m_{\|S_1\|,\ldots,\|S_k\|},$$
with $S_i$ as in Theorem \ref{thm4} or Proposition \ref{biglem}.
\begin{prop}\label{biglem} For $\ell\ge k$, the only $k$-tuples $(n_1,\ldots,n_k)$ that can be decomposed in an odd number of ways as $n_i=s_i+t_i$ with $(t_1,\ldots,t_k)$ a permutation of $(1,2,4,\ldots,2^{k-1})$ and $s_i=\|S_i\|$, where $S_1,\ldots,S_k$ is a partition of $\{1,2,4,\ldots,2^{\ell-1}\}$ into $k$ nonempty subsets, are the permutations of $(2,4,8,\ldots,2^{k-1},2^\ell)$.\end{prop}
\begin{proof} We will show that all
\begin{equation}\label{s,t}\begin{pmatrix}S_1&\cdots&S_k\\ t_1&\cdots&t_k\end{pmatrix}\end{equation}
as in the proposition can be grouped into pairs with equal column sums $(\|S_1\|+t_1,\ldots,\|S_k\|+t_k)$ except for permutations (by column) of
\begin{equation}\label{spec}\begin{pmatrix}2^0&2^1&\cdots&2^{k-2}&\{2^{k-1},\ldots,2^{\ell-1}\}\\ 2^0&2^1&\cdots&2^{k-2}&2^{k-1}\end{pmatrix}.\end{equation}
It is easy to see that (\ref{spec}) is the only matrix (\ref{s,t}) with its column sum.

Let $M$ be  a matrix (\ref{s,t}), and let $K=\{2^0,\ldots,2^{k-1}\}$. Define $f:K\to K$ by $f(x)=t_i$ if $x\in S_i$. Since $f^{i+1}K\subseteq f^iK$, there is a smallest nonnegative integer $N$ such that $f^{N+1}K=f^NK$; i.e., with $T:=f^NK$, $f|T$ is an automorphism of $T$.

{\bf Case 1}: $f|T\ne1_T$.  We pair $M$ with the matrix obtained by interchanging $x$ and $f(x)$ in all columns with $t_i\in T$. Note that this preserves column sums and is involutive, in the sense that the new matrix is also of Case 1 type, and would lead to $M$. For example,
$$\begin{pmatrix}2^0&\{2^1,2^3\}&2^2\\ 2^1&2^0&2^2\end{pmatrix}\text{ is paired with }\begin{pmatrix}2^1&\{2^0,2^3\}&2^2\\ 2^0&2^1&2^2\end{pmatrix}.$$

{\bf Case 2}: $f|T=1_T$. Let $2^i\in T$ be minimal such that the $S_j$ above it in (\ref{s,t}) {\it strictly} contains $\{2^i\}$. Such an $i$ must exist since either $T=K$ or else some $\ell\in K-T$ must satisfy $f(\ell)\in T$.

{\bf Case 2a}: $i<k-1$.
 Then
$$\begin{pmatrix}\cdots&\{2^i,D\}&\cdots&E&\cdots\\ \cdots&2^i&\cdots&2^{i+1}&\cdots\end{pmatrix}\text{ is paired with }\begin{pmatrix}\cdots&D&\cdots&\{2^i,E\}&\cdots\\ \cdots&2^{i+1}&\cdots&2^i&\cdots\end{pmatrix}.$$
Here $D$ and $E$ represent nonempty collections of 2-powers.
% Note that it is possible here (if $2^{i+1}\in E$ or $D$) that one of the functions $f$ here could be the identity. Indeed, if $f$ equals the identity and some $2^t$ with $t<k-1$ is accompanied in the top row, it will fall into this case.

{\bf Case 2b}: $i=k-1$. Let $S_v$ be the set above $2^{k-1}$ in (\ref{s,t}). If $S_v$ contains $\{2^i:\ k-1\le i\le\ell-1\}$, then the matrix must be of the form (\ref{spec}), since $f$ must be bijective, and hence $T=K$ and $f=1_K$. Otherwise,
let $2^e$ be the smallest 2-power $\ge 2^k$ not in $S_v$.  There is a sequence $2^{i_1},\ldots,2^{i_r}$ such that $2^{i_1}$ lies below $2^e$ in $M$, $f(2^{i_j})=2^{i_{j+1}}$ for $1\le j<r$, and $2^{i_r}\in S_v$. This sequence of $2^{i_j}$'s must eventually be in $S_v$ because otherwise it would have a cycle, and be in Case 1. The matrix $M$ is paired with one in which all the $2^{i_j}$'s are moved up or down within their column, while the $2^j$'s with $k-1\le j\le e$ are interchanged between the columns
containing the $2^e$ and the $2^{k-1}$, with other entries remaining fixed. We illustrate with a case $r=2$, $e=k+2$.
$$\begin{pmatrix}2^{k+2}&\cdots&2^{t_1}&\cdots&\{2^{t_2},2^{k-1},2^k,2^{k+1}\}\\ 2^{t_1}&\cdots&2^{t_2}&\cdots&2^{k-1}\end{pmatrix}\leftrightarrow\begin{pmatrix}\{2^{t_1},2^{k-1},2^k,2^{k+1}\}&
\cdots&2^{t_2}&\cdots&2^{k+2}\\ 2^{k-1}&\cdots&2^{t_1}&\cdots&2^{t_2}\end{pmatrix}.$$
\end{proof}

\section{Relations with Dickson invariants}\label{Dicksec}
In this section, we discuss the relationship between our polynomials $p_{\ell,j}$ and the Dickson invariants. Most of the results in this section were suggested by a referee.

Let $V$ be an $\F_2$-vector space with basis $x_1,\ldots,x_k$, and $S(V)$ its symmetric algebra. The general linear group GL$(V)$ acts on $S(V)$, and the ring of invariant elements is called the 2-primary Dickson algebra $D_k$. Dickson showed in \cite{Dick} that $D_k$ is a polynomial algebra on classes $c_{j}$ of grading $2^k-2^j$ for $0\le j\le k-1$. We suppress the usual $k$ from the subscript, as we did with our $p$'s, since it is fixed throughout this paper.

If $M$ is a Vandermonde determinant in $x_1,\ldots,x_k$ with distinct 2-power exponents, then $M$ is invariant under the action of GL$(V)$. This is easily proved using linearity of determinants and that $(\sum\a_ix_i)^{2^t}=\sum\a_ix_i^{2^t}$. Since our polynomials $p_{\ell,j}$ in (\ref{pdef}) are ratios of Vandermonde determinants with distinct 2-power exponents, they are elements of $D_k$, and one might seek to express them in terms of the generators $c_{j}$.

Our first result is that our polynomials $p_{k,j}$ (i.e., those with $\ell=k$) are exactly the generators $c_{j}$.
\begin{prop} For $0\le j\le k-1$, $p_{k,j}=c_{j}$.\label{kjprop}\end{prop}
\begin{proof} By \cite[Prop 1.3a]{Wilk}, $c_{j}=\dfrac{m_{2^0,\ldots,\widehat{2^j},\ldots,2^{k-1},2^k}}{m_{2^0,\ldots,2^{k-1}}}$, which by (\ref{pdef}) equals $p_{k,j}$.\end{proof}

The following corollary is now immediate from Theorem \ref{thm3}.
\begin{cor} \label{Dickcor} The Dickson invariant usually called $c_{k,j}$ over $\F_2$ is the sum of all monomials of degree $2^k-2^j$ in $x_1,\ldots,x_k$ in which all nonzero exponents are  $2$-powers.
\end{cor}
\ni This result was certainly  known to some, but we could not find it explicitly stated in the literature. One place that essentially says it is \cite[Prop 3.6(c)]{Ar}.

Some of our elements $p_{\ell,j}$ are related to one another in the following way.
\begin{prop} For $\ell\ge k+1$, we have $p_{\ell,0}=c_{0}p_{\ell-1,k-1}^2$. In particular, $p_{k+1,0}=c_{0}c_{k-1}^2$.\end{prop}
\begin{proof} The denominator in (\ref{pdef}) equals $c_{0}$, so we have
$$c_{0}p_{\ell,0}=m_{2^1,\ldots,2^{k-1},2^\ell}=m^2_{2^0,\ldots,2^{k-2},2^{\ell-1}}=c_{0}^2p_{\ell-1,k-1}^2.$$
The second part follows from Proposition \ref{kjprop}.\end{proof}

There is an action of the mod-2 Steenrod algebra on $S(V)$ and on $D_k$, and the following complete formula was obtained in \cite{Hung}.
\begin{prop} \label{hprop} $($\cite{Hung}$)$ In the Dickson algebra $D_k$, for $0\le s\le k-1$,
$$\sq^ic_s=\begin{cases}c_r&i=2^s-2^r\\ c_rc_t&i=2^k-2^t+2^s-2^r,\ r\le s<t\\ c_s^2&i=2^k-2^s\\ 0&\text{otherwise.}\end{cases}$$
\end{prop}

 Without using that formula, we can easily obtain the following result.
\begin{prop} For $0\le j\le k-1$,
$$\sq^{2^i}p_{\ell,j}=\begin{cases}p_{\ell,j-1}&i=j-1\\ 0&i\ne j-1,\ i<k-1.\end{cases}$$\label{prop1}
\end{prop}
\begin{proof} For $i<k-1$, $\sq^{2^i}(m_{2^0,\ldots,2^{k-1}})=0$ since each term with factor $\sq^{2^i}(x_s^{2^i})x_t^{2^{i+1}}$ is paired with an equal term $x_s^{2^{i+1}}\sq^{2^i}(x_t^{2^i})$. Using the Adem relations, it follows that $\sq^n(m_{2^0,\ldots,2^{k-1}})=0$ for $0<n<2^{k-1}$. Similarly, for $0<i<k-1$,
$$\sq^{2^i}(m_{2^0,\ldots,\widehat{2^j},\ldots,2^{k-1},2^\ell})=\begin{cases}m_{2^0,\ldots,\widehat{2^{j-1}},\ldots,2^{k-1},2^\ell}&i=j-1\\ 0&\text{otherwise.}\end{cases}$$
The result follows from applying the Cartan formula to
$$m_{2^0,\ldots,2^{k-1}}p_{\ell,j}=m_{2^0,\ldots,\widehat{2^j},\ldots,2^{k-1},2^\ell}.$$
\end{proof}

This result meshes nicely with the following one.
\begin{prop} For $\ell\ge k$,
$$p_{\ell+1,k-1}=\sum_j c_{j}\sq^{2^\ell-2^k+2^j}p_{\ell,k-1}.$$\label{prop2}\end{prop}
\begin{proof} We have
\begin{equation}\label{ceq}\sq^{2^\ell}(c_{0}p_{\ell,k-1})=\sq^{2^\ell}(m_{2^0,\ldots,2^{k-2},2^\ell})=m_{2^0,\ldots,2^{k-2},2^{\ell+1}}=c_{0}p_{\ell+1,k-1}.\end{equation}
As a special case of Proposition \ref{hprop}, we have, for $i>0$,
$$\sq^ic_{0}=\begin{cases}c_{j}c_{0}&i=2^k-2^j\\ 0&\text{otherwise.}\end{cases}$$
Applying the Cartan formula to the LHS of (\ref{ceq}) and cancelling $c_{0}$ yields the result.\end{proof}

In principle, iterating Propositions \ref{prop1} and  \ref{prop2} enables us to obtain complete formulas expressing our polynomials $p_{\ell,j}$ in terms of the $c_i$'s. For $\ell=k$, this was initiated in our Proposition \ref{kjprop}. Here we do it for $\ell=k+1$ and $k+2$. For $\ell\ge k+3$, the formulas become unwieldy.
\begin{thm}\label{k+12} For $0\le j\le k-1$,
\begin{itemize}
\item[a.] $p_{k+1,j}=c_{j-1}^2+c_jc_{k-1}^2$;
\item[b.] $p_{k+2,j}=c_jc_{k-2}^4+c_jc_{k-1}^6+c_{j-1}^2c_{k-1}^4+c_{j-2}^4$.
\end{itemize}
\end{thm}
\begin{proof}
(a). By Propositions \ref{kjprop}, \ref{prop2}, and \ref{hprop}, we have
\begin{eqnarray*}p_{k+1,k-1}&=&\sum_j c_j\sq^{2^j}p_{k,k-1}=\sum_j c_j\sq^{2^j}c_{k-1}\\
&=&c_{k-2}^2+c_{k-1}\cdot c_{k-1}^2.\end{eqnarray*}
Assume the result true for $j$. By Propositions \ref{prop1} and \ref{hprop},
\begin{eqnarray*}p_{k+1,j-1}&=&\sq^{2^{j-1}}(c_{j-1}^2+c_jc_{k-1}^2)\\
&=&(\sq^{2^{j-2}}c_{j-1})^2+\sum_m(\sq^{2^{j-1}-2m}c_j)(\sq^mc_{k-1})^2\\
&=&c_{j-2}^2+(\sq^{2^{j-1}}c_j)c_{k-1}^2\\
&=&c_{j-2}^2+c_{j-1}c_{k-1}^2,\end{eqnarray*}
extending the induction.

(b). Applying Proposition \ref{prop2} to part (a),  we obtain
\begin{eqnarray*} p_{k+2,k-1}&=&\sum_jc_j\sq^{2^k+2^j}(c_{k-2}^2+c_{k-1}^3)\\
&=&\sum_jc_j(\sq^{2^{k-1}+2^{j-1}}c_{k-2})^2+\sum_{j,m}c_j(\sq^{2^k+2^j-2m}c_{k-1})(\sq^mc_{k-1})^2.\end{eqnarray*}
Using Proposition \ref{hprop}, the first sum equals $c_{k-2}c_{k-3}^2c_{k-1}^2+c_{k-1}c_{k-2}^4$, while the second equals
$$c_{k-3}^4+c_{k-2}c_{k-1}^2c_{k-3}^2+c_{k-2}^2c_{k-1}^4+c_{k-1}^7.$$
Combining these yields the result for $j=k-1$. The result for arbitrary $j$ follows by decreasing induction on $j$, similarly to part (a).
\end{proof}

 \def\line{\rule{.6in}{.6pt}}

\end{document}